\documentclass[a4paper,12pt]{article}

\usepackage{amsmath, wrapfig}
\usepackage{dsfont, a4wide, amsthm, amssymb, amsfonts, graphicx}
\usepackage{fancyhdr, xspace, psfrag, setspace, supertabular, color, enumerate}
\usepackage{hyperref}

\newtheorem{theorem}{Theorem}[section]
\newtheorem{proposition}[theorem]{Proposition}
\newtheorem{lemma}[theorem]{Lemma}

\newtheorem{corollary}[theorem]{Corollary}

\newtheorem{definition}[theorem]{Definition}

\newtheorem{example}[theorem]{Example}
\newtheorem{question}[theorem]{Question}

\def\cp{\mathcal{P}}

\def\F{\mathbb{F}}

\def\N{\mathbb{N}}

\DeclareMathOperator{\Mc}{Mc}
\DeclareMathOperator{\Monec}{M_1c}
\DeclareMathOperator{\Mtwoc}{M_2c}
\DeclareMathOperator{\Monetwoc}{M_1 \wedge M_2c}

\DeclareMathOperator{\sr}{sr}

\title{Simple proofs for lattice coverings and sparse tensors}
\author{Thomas Karam\footnote{Mathematical Institute, University of Oxford. Email: \texttt{thomas.karam@maths.ox.ac.uk}.}}

\begin{document}
\maketitle

\begin{abstract}
We provide simple proofs of analogues for coverings numbers of lattices of several recently studied basic statements on the ranks of tensors. We highlight the differences and analogies between the proofs in both settings.
\end{abstract}

\tableofcontents

\section{Introduction}

Throughout, for any integer $d \ge 1$ we will take $n_1, \dots, n_d$ to be positive integers. All our statements will be uniform in $n_1, \dots, n_d$ (but not in $d$) unless stated otherwise. We will denote a field by $\F$, and all our statements will be uniform with respect to $\F$. An order-$d$ \emph{lattice subset} will be taken to be a subset of $[n_1] \times \dots \times [n_d]$, and an order-$d$ \emph{tensor} will be taken to be a function $[n_1] \times \dots \times [n_d] \to \F$ (and more generally a function $X_1 \times \dots \times X_d \to \F$ where $X_1, \dots, X_d$ are finite subsets of $\N$). For any positive integer $k$ we will denote by $[k]$ the set $\{1, \dots, k\}$ of positive integers up to $k$.

The present paper will involve several notions of richness of lattice subsets and of tensors. In the case of lattice subsets, perhaps the simplest such notion is their size. Another is the smallest number of sets of a specified type which together suffice to cover the lattice subset. That notion specialises to the size if the specified sets for the covering are all the singletons. In the case of tensors, arguably the simplest well-known notion is the matrix rank, which in turn generalises to a wide class of notions of rank for $d \ge 3$.

There is a simple connection between order-$d$ lattice subsets and tensors: given an order-$d$ tensor $T$, the support \[\{Z(T):= \{(x_1, \dots, x_d) \in [n_1] \times \dots \times [n_d]: T(x_1, \dots, x_d) \neq 0 \}\] of $T$ is an order-$d$ lattice subset. Conversely an order-$d$ lattice subset may be viewed as an order-$d$ tensor with entries in $\{0,1\}$. We now formally define the class of covering numbers of lattice subsets which we will be concerned with in this paper.

\begin{definition}

Let $d \ge 1$ be an integer. We make the following definitions.

If $B$ is a subset of $[d]$, then we say that a \emph{$B$-subspace} is a maximal subset $S$ of $[n_1] \times \dots \times [n_d]$ such that $x_i = y_i$ whenever $x,y \in S$ and $i \in B^c$, and say that it has \emph{order} $|B|$.

If $M$ is a non-empty subset of $\cp([d])$, then we say that an $M$-subspace is a $B$-subspace for some $B \in M$. 

If $A$ is furthermore a subset of $[n_1] \times \dots \times [n_d]$ then we say that $A$ has $M$-covering number at most $1$ if $A$ is contained in an $M$-subspace.

The $M$-covering number of $A$, denoted by $\Mc(A)$, is the smallest nonnegative integer $k$ such that we can write $A$ as a union \begin{equation} A \subset A_1 \cup \dots \cup A_k \label{M-covering decomposition} \end{equation} of $k$ subsets $A_1, \dots, A_k$ of $[n_1] \times \dots \times [n_d]$ which each have $M$-covering number at most $1$.

We will refer to an inclusion such as \eqref{M-covering decomposition} as an \emph{$M$-covering decomposition}, and to the integer $k$ inside such an inclusion as the \emph{length} of that decomposition.

\end{definition}

For instance, for $d=2$, a $\{1\}$-subspace is a horizontal line, a $\{2\}$-subspace is a vertical line, and any set contained in $\{x=1\} \cup \{y=1\}$ has $\{\{1\},\{2\}\}$-covering number at most $2$. For any $d \ge 1$, a $\emptyset$-subspace is a point, and the $\{\emptyset\}$-covering number of $A$ is equal to the size of $A$. At the other extreme, a $[d]$-subspace is the whole of $[n_1] \times \dots \times [n_d]$, and any subset $A$ has $\{[d]\}$-covering number at most $1$.

As for tensors, there is a somewhat analogous class of notions which also arises from assigning different roles to different coordinates, although there is extra flexibility which led us (and has led us in previous work \cite{K. subtensors}, \cite{K. interplay}) to consider a class which is indexed by families of partitions of $[d]$ rather than families of subsets of $[d]$. This class had also arisen in the paper of Naslund \cite{Naslund} which originally introduced the partition rank of tensors.

\begin{definition} \label{Rrk definition} Let $d \ge 1$ be an integer, and let $R$ be a non-empty family of partitions of $\lbrack d \rbrack$. We say that an order-$d$ tensor $T$ has \emph{$R$-rank at most $1$} if there exists a decomposition \[ T(x_1,\dots ,x_d) = \prod_{J \in P} a_J(x(J)) \] of $T$ that holds for every $(x_1,\dots,x_d) \in [n_1] \times \dots \times [n_d]$, where $P \in R$ is a partition of $[d]$, and $a_J: \prod_{j \in J} [n_j] \rightarrow \mathbb{F}$ is a function for each $J \in P$. We say that the \emph{$R$-rank} of $T$ is the smallest nonnegative integer $k$ such that we can write $T$ as a sum of $k$ order-$d$ tensors each with $R$-rank at most $1$. \end{definition}

After the usual rank of matrices, the notion of $R$-rank which has received the most attention is the tensor rank, which corresponds to the case where $R$ contains only the discrete partition $\{\{1\}, \dots, \{d\}\}$, due to its applications in computational complexity theory. The notion of rank which has the strongest known connections to some covering number is the slice rank, which corresponds to \[R = \{\{\{1\},\{1\}^c\}, \{\{2\},\{2\}^c\}, \dots, \{\{d\},\{d\}^c\}\}\] and was formulated by Tao in 2016 \cite{Tao} when rewriting the proof of the breakthrough by Croot-Lev-Pach \cite{Croot Lev Pach} and then Ellenberg and Gijswijt \cite{Ellenberg Gijswijt} on the cap-set problem in a more symmetric way. Meanwhile, an appealing analogue to the slice rank on the side of covering numbers is the $M$-covering number where $M = \{\{j\}^c: j \in [d]\}$. We shall refer to this quantity as the \emph{slice covering number}. After the slice rank was introduced, Sawin and Tao were able to obtain lower bounds on the slice rank of a tensor in terms of slice covering numbers. In one case, they discovered an exact identity between the two, which follows immediately from their more general statement \cite{Sawin and Tao}, Proposition 4.

If $d \ge 1$ is an integer, and $A$ is an order-$d$ lattice subset, then we say that $A$ is an \emph{antichain} if no two distinct elements $x,y \in A$ satisfy $x_j \le y_j$ for every $j \in [d]$.

\begin{proposition} \label{Sawin-Tao connection}

Let $d \ge 2$ be an integer, and let $T$ be an order-$d$ tensor. If the support $Z(T)$ of $T$ is contained in an antichain, then the slice rank of $T$ is equal to the slice covering number of $Z(T)$.

\end{proposition}

Proposition \ref{Sawin-Tao connection} cannot be extended to arbitrary tensors: for instance a tensor with entries all equal to $1$ has slice rank equal to $1$, but may have unbounded slice covering number. Nonetheless, the class of tensors covered by Proposition \ref{Sawin-Tao connection} is already wide enough to have combinatorial applications. In particular, as the remainder of the blog post of Sawin and Tao shows, it is possible to use this characterisation to show the sharpness of the slice rank bounds for both the capset problem (\cite{Sawin and Tao}, Proposition 9) and the result \cite{Naslund and Sawin} of Naslund and Sawin on sunflower-free sets (\cite{Sawin and Tao}, Proposition 10).

In this paper we give simple proofs of analogues of the main results of previous works \cite{Gowers}, \cite{K. subtensors}, \cite{K. decompositions}, \cite{K. interplay}, for slice coverings numbers of lattices rather than ranks of tensors, and compare the difficulties and proofs of the statements in both settings.

How close the analogy will be between the cases of the covering numbers of lattice subsets and of the ranks of tensors will depend on the specific notions involved and on the question at hand. Proposition \ref{Sawin-Tao connection} shows that there are situations where this analogy is strong to the extent that we have an exact correspondence between the two. As we will explain throughout the remainder of the paper, in some other situations the proofs will be simpler for covering numbers, but there will be still strong analogies in part of the proof methods. And finally in yet some other cases the proofs for covering numbers will be considerably simpler than and will have little resemblance with the proofs for ranks.

\section*{Acknowledgements}

The author thanks Timothy Gowers for a discussion in 2021 where the suggestion arose that proving statements on covering numbers of lattices could be used as a warm-up for their more difficult tensor versions. The author also thanks the organisers of the European Conference on Combinatorics, Graph Theory and Applications 2023 (Eurocomb’23) for encouragement after the conference which led to the present paper being put together.

\section{Main results}

We devote the present section to our main results regarding covering numbers. We begin with a property for which the proofs contain major difficulties for the ranks of tensors, but are immediate for covering numbers of lattice subsets. 

For any integer $d \ge 1$, we say that lattice subsets $A^1$, $A^2$ are in diagonal sum if we can write \[A_1 \subset X_1^1 \times \dots \times X_d^1 \text{ and } A^2 \subset X_1^2 \times \dots \times X_d^2\] for some $X_1^1, X_1^2 \subset [n_1]$, \dots, $X_d^1, X_d^2 \subset [n_d]$ such that the intersection $X_j^1 \cap X_j^2$ is empty for every $j \in [d]$. We likewise say that order-$d$ tensors are in diagonal sum if their supports are in diagonal sum.

Whether the $R$-rank of tensors is additive over diagonal sums for a given $R$ is still an open question in general. In particular that question is still open for the partition rank. For the tensor rank, that question remained for a long time a conjecture of Strassen until it was negatively settled by Shitov \cite{Shitov} in 2019. It was then only in 2021 that Gowers \cite{Gowers} proved diagonal additivity for the slice rank. On the other hand, diagonal additivity is essentially always satisfied for covering numbers of lattice subsets.

\begin{proposition} \label{Diagonal sum for lattice coverings}

Let $d \ge 1$ be an integer, and let $M \subset \cp([d])$ be non-empty. If $A_1$ and $A_2$ are order-$d$ lattice subsets in diagonal sum and $M$ does not contain $[d]$, then \[\Mc(A_1 \cup A_2) = \Mc(A_1) + \Mc(A_2).\]

\end{proposition}

If on the other hand $M$ contains $[d]$, then this identity becomes false, since the $M$-covering number is then always at most $1$. Proposition \ref{Diagonal sum for lattice coverings} is immediate, because a given $B$-subspace for any $B \neq [d]$ cannot simultaneously contain points of $A_1$ and of $A_2$, whereas a single tensor with $R$-rank equal to $1$ (for any notion of $R$-rank) can have an effect on both diagonal blocks at once.

Another topic where the proofs are much simpler for lattice covering numbers than for ranks of tensors is that of the interplay between the different notions. In \cite{K. interplay} we addressed the following question: if the $R_1, \dots, R_k$-ranks of a tensor are all bounded, then for which $R$ is the $R$-rank of that tensor necessarily bounded ? Proposition \ref{interplay between covering numbers} answers the analogous question for covering numbers.

If $d \ge 1$ is an integer, and $M_1, M_2 \subset \cp([d])$ are non-empty then we write \[M_1 \wedge M_2 = \{B_1 \cap B_2: B_1 \in M_1, B_2 \in M_2\}.\] 

\begin{proposition} \label{interplay between covering numbers}

Let $d \ge 1$ be an integer. \begin{enumerate}

\item Let $M \subset \cp([d])$ be non-empty and let $B \in \cp([d])$. The $M$-covering number of any $B$-subspace is equal to \[\min_{B’ \in M} \prod_{j \in B \setminus B’} n_j.\] In particular, if $n_1=\dots=n_d$ are all equal to some common integer $n$, then this quantity is equal to \[n^{\min_{B’ \in M} |B \setminus B’|},\] which is equal to $1$ if $B$ is contained in some $B' \in M$ and is at least $n$ otherwise.

\item Let $M_1, M_2 \subset \cp([d])$ be non-empty, and let $k_1, k_2$ be nonnegative integers. If $A$ is an order-$d$ lattice subset such that $\Monec A \le k_1$ and $\Mtwoc A \le k_2$ then $\Monetwoc A \le k_1k_2$.

\end{enumerate} \end{proposition}

\begin{proof}

We first prove (i). If $S$ is a $B$-subspace and $S’$ is a $B’$-subspace for some $B’ \in M$, then we have \[|S| / |S \cap S’| = \prod_{j \in B \setminus B’} n_j,\] which shows the lower bound. Choosing $B’$ which minimises this expression and then considering the $M$-covering of $B$ consisting only of $B’$-subspaces then provides the upper bound.

We now prove (ii). We begin by writing respective $M_1$- and $M_2$-covering decompositions \[A \subset S_1^1 \cup \dots \cup S_{k_1}^1 \text{ and } A \subset S_1^2 \cup \dots \cup S_{k_2}^2 \] of $A$, where each set $S_i^1$ is an $M_1$-subspace and each set $S_i^2$ is an $M_2$-subspace. It follows that $A$ is a subset of \[\bigcup_{1 \le i_1 \le k_1} \bigcup_{1 \le i_2 \le k_2} S_{i_1}^1 \cap S_{i_2}^2.\] For any $i_1 \in [k_1]$ and every $i_2 \in [k_2]$ the intersection $S_{i_1}^1 \cap S_{i_2}^2$ is a $B_{i_1} \cap B_{i_2}$-subspace for some $B_{i_1} \in M_1$ and $B_{i_2} \in M_2$; in other words it is an $M_1 \wedge M_2$-subspace. \end{proof}

Let $X_1 \subset [n_1], \dots, X_d \subset [n_d]$ be coordinate subsets. If $A$ is an order-$d$ lattice subset then we write \[A(X_1 \times \dots \times X_d) = A \cap (X_1 \times \dots \times X_d)\] for the restriction of $A$ to the product $X_1 \times \dots \times X_d$. Likewise, if $T$ is an order-$d$ tensor then we write $T(X_1 \times \dots \times X_d)$ for the tensor $T’: X_1 \times \dots \times X_d \to \F$ defined by \[T’(x_1, \dots, x_d) = T(x_1, \dots, x_d)\] for every $(x_1, \dots, x_d) \in [n_1] \times \dots \times [n_d]$.

The main results which we will focus on throughout the remainder of this paper are the following four. The first is an analogue for coverings of lattice subsets of \cite{K. subtensors}, Theorem 1.4, which informally states that a high-rank tensor must necessarily have a high-rank subtensor with bounded size.

\begin{theorem}\label{linear bound for lattice subsets}

Let $d,l \ge 1$ be integers, and let $M \subset \cp([d])$ be non-empty. If $A$ is an order-$d$ lattice subset such that $\Mc(A) \ge |M|l$, then there exist $X_1 \subset [n_1], \dots, X_d \subset [n_d]$ each with size at most $l$ such that \[\Mc A(X_1 \times \dots \times X_d) \ge l.\]

\end{theorem}

In the case of tensors, linear bounds are not known. In \cite{K. subtensors} it was conjectured (Conjecture 13.1) that we can take linear bounds for both the rank of the original tensor and the sizes of the sets $X_1, \dots, X_d$. Thereafter, motivated by an application \cite{Briet Buhrman Castro-Silva Neumann} to limitations on the decoding of corrupted classical and quantum error-correcting codes with $\mathrm{NC}^0[\oplus]$ circuits, Bri\"et and Castro-Silva also asked in \cite{Briet Castro-Silva}, Section 4 for linear bounds (referred to as the “linear core property” in that paper) for a class of notions of rank called “natural rank functions” (which neither contains nor is contained in the class of notions of $R$-ranks).

Motivated by an extension of a theorem of Green and Tao \cite{Green and Tao} on the equidistribution of high-rank polynomials to subsets of finite prime fields, ultimately realised in \cite{Gowers and K equidistribution}, another statement involving subtensors was proved in \cite{K. subtensors}, Theorem 1.8: if a tensor has high rank after any changes of its entries that do not have their coordinates all pairwise distinct, then it has a high-rank subtensor consisting entirely of entries with pairwise distinct coordinates. There again, linear bounds were conjectured in \cite{K. subtensors}, Conjecture 13.6. We show that an analogous result holds for lattice subsets and prove linear bounds there.

If $d \ge 1$ is an integer then we write $E$ for the set of points of $[n_1] \times \dots \times [n_d]$ which do \emph{not} have all their coordinates pairwise distinct.

\begin{theorem}\label{off-diagonal covering number}

Let $d,l \ge 1$ be integers, and let $M \subset \cp([d])$ be non-empty. If $A$ is an order-$d$ lattice subset such that $\Mc(A \setminus E) \ge d^d|M|l$ then there exist $X_1 \subset [n_1], \dots, X_d \subset [n_d]$ pairwise disjoint such that \[\Mc A(X_1 \times \dots \times X_d) \ge l.\]

\end{theorem}

In a third direction, it was proved (\cite{K. decompositions}, Theorem 3.5) that up to a simple class of transformations, a tensor of a given order with bounded slice rank over a finite field with bounded size necessarily has at most a bounded number of minimal-length slice rank decompositions, up to a simple class of transformations (which is indispensable in such a statement). We prove an analogue of this statement for $M$-covering of lattice subsets for arbitrary $M$, which is also simpler in that it is not necessary to appeal to any analogue of the previously mentioned class of transformations.

\begin{theorem}\label{number of decompositions}

Let $d,l \ge 1$ be integers, and let $M \subset \cp([d])$ be non-empty. If $A$ is an order-$d$ lattice subset such that $\Mc(A) = l$ then there are at most $(|M|l^d)^l$ $l$-tuples $(S_1, \dots, S_l)$ such that \[A \subset S_1 \cup \dots \cup S_l\] and the set $S_i$ is an $M$-subspace for every $i \in [l]$.

\end{theorem}

The following example shows that the bound obtained in Theorem \ref{number of decompositions} is not too far from the optimal bound. Indeed for fixed $d$ and $M$ both bounds are of the type $\exp(C(d,M) l \log l)$, where the two constants $C(d,M)$ are away by a factor which is at most linear in $d$.

\begin{example}

Let $d,l \ge 1$ be integers. Assume that $\min(n_1, \dots, n_d) \ge l$. We consider the order-$d$ lattice subset $A = \{(i,\dots,i): i \in [l]\}$. Let $M$ be a non-empty subset of $\cp([d])$ which does not contain $[d]$. No two points of $A$ can be contained in the same $M$-subspace, so the $M$-covering number of $A$ is equal to $l$. For each point $x \in A$, there are $|M|$ possible choices for an $M$-subspace containing $x$, so the number of possible $l$-tuples $(S_1, \dots, S_l)$ is equal to $l! |M|^l$. \end{example}

Coming full circle, we deduce from Theorem \ref{number of decompositions} another version of Theorem \ref{linear bound for lattice subsets}, which is stronger than Theorem \ref{linear bound for lattice subsets} in that the covering number of the restricted lattice subset is the same as that of the original lattice subset, but weaker in that the sizes of the coordinate sets are no longer linear in (either) covering number. 

In the following statement and in Corollary \ref{Corollary on the slice rank of sparse tensors}, the asymptotic notation $O$ is used with $d,M$ fixed and $l$ large.

\begin{theorem}\label{lattice subset with same covering number}

Let $d,l \ge 1$ be integers, and let $M \subset \cp([d])$ be non-empty. If $A$ is an order-$d$ lattice subset such that $\Mc(A) \ge l$ then there exist $X_1 \subset [n_1], \dots, X_d \subset [n_d]$ each with size at most $O(\max(|M|,d)^{dl})$ such that \[\Mc A(X_1 \times \dots \times X_d) \ge l.\]

\end{theorem}

In the case where the base field $\F$ is finite a qualitative analogue of Theorem \ref{lattice subset with same covering number} for the ranks of tensors was proved by Blatter, Draisma, and Rupniewski (\cite{Blatter Draisma Rupniewski}, Corollary 1.3.3) through an analogue (\cite{Blatter Draisma Rupniewski}, Theorem 1.2.1) of the graph minor theorem to tensors, which informally states that any sequence $(T_i)$ of tensors over a fixed finite field is such that we can eventually find $i<j$ such that $T_i$ is a restriction of $T_j$ up to linear transformations.

It is also known that we cannot in general simultaneously require the covering number of the restriction and the sizes of the coordinate sets to be equal to the covering number of the original lattice subset. An order-$3$ lattice subset with slice covering number equal to $4$, but such that all its $4 \times 4 \times 4$ restrictions have slice covering number equal to $3$ was constructed by Gowers (\cite{K. subtensors}, Proposition 3.1). Together with the Sawin-Tao connection, Proposition \ref{Sawin-Tao connection}, this construction immediately led to an analogous counterexample for the slice rank of tensors as well.

Specialising Theorem \ref{linear bound for lattice subsets}, Theorem \ref{off-diagonal covering number} and Theorem \ref{lattice subset with same covering number} to the slice covering number and using Proposition \ref{Sawin-Tao connection} we obtain corresponding statements for the slice rank of tensors which are sufficiently sparse in the sense that they are supported inside an antichain. We write $\sr T$ for the slice rank of a tensor $T$.

\begin{corollary} \label{Corollary on the slice rank of sparse tensors}

Let $d \ge 2, l \ge 1$ be integers, and let $T$ be an order-$d$ tensor supported inside an antichain. \begin{enumerate}[(i)]

\item If $\sr T \ge dl$ then we can find $X_1 \subset [n_1], \dots, X_d \subset [n_d]$ with size at most $l$ such that \[\sr T(X_1 \times \dots \times X_d) \ge l.\]

\item If for every order-$d$ tensor $V$ supported inside $E$ we have \[\sr (T-V) \ge d^{d+1} l\] then we can find $X_1 \subset [n_1], \dots, X_d \subset [n_d]$ pairwise disjoint such that \[\sr T(X_1 \times \dots \times X_d) \ge l.\]

\item If $\sr T \ge l$ then then we can find $X_1 \subset [n_1], \dots, X_d \subset [n_d]$ with size at most $O(\max(|M|,d)^{dl})$ such that \[\sr T(X_1 \times \dots \times X_d) \ge l.\]

\end{enumerate}

\end{corollary}

The remainder of the paper is organised as follows. In Section \ref{Section: Independent subsets arguments} we prove Theorem \ref{linear bound for lattice subsets} and Theorem \ref{off-diagonal covering number}. There, the main tool that we use, which does not appear to have any simple analogue for tensors, is a lemma which states that from any lattice subset $A$ with high $M$-covering number we are able to find a large subset $A’ \subset A$ of points that are independent in the sense that any $M$-subspace can only contain at most one point of $A’$. After that, in Section \ref{Section: Descending subspaces arguments} we prove Theorem \ref{number of decompositions} and then Theorem \ref{lattice subset with same covering number} successively. Finally, in Section \ref{Section: Further directions} we discuss some remaining directions of progress.

\section{Independent subsets arguments} \label{Section: Independent subsets arguments}

\subsection{Finding a subset of independent points}

We begin by introducing the key ingredient of this section, which makes the proofs of Theorem \ref{linear bound for lattice subsets} and of Theorem \ref{off-diagonal covering number} much simpler for covering numbers than for ranks of tensors.

If $d \ge 1$ is an integer, $M \subset \cp([d])$ is non-empty, and $A$ is an order-$d$ lattice subset, then we write $I_M(A)$ for the \emph{$M$-independence number of $A$}, defined as the largest size of a subset $A’ \subset A$ satisfying the property that whenever $x,y$ are two distinct elements of $A’$ there is no $M$-subspace which contains both $x$ and $y$. If $A’$ is a largest such subset (or in fact any such subset), then in particular any further subset $A''$ of $A'$ satisfies $\Mc(A'') = |A''|$.

\begin{lemma} \label{subset where elements cannot be in the same subspace}

Let $d \ge 1$, $l \ge 1$ be integers, let $M \subset \cp([d])$ be non-empty, and let $A$ be an order-$d$ lattice subset with $M$-covering number at least $l$. Then $I_M(A) \ge |A|/|M|$.

\end{lemma}

\begin{proof}

If $u$ is an element of $[n_1] \times \dots \times [n_d]$, then we let $M(u)$ be the union of all $M$-subspaces of $[n_1] \times \dots \times [n_d]$ containing $u$. We note that $M(u)$ necessarily has $M$-covering number at most $|M|$. Given an order-$d$ lattice subset $A$, we construct $A’$ as follows. We first choose arbitrarily an element $u^1 \in A$. We then choose arbitrarily another element $u^2 \in A \setminus M(u_1)$ if we can do so. We continue inductively as long as we can: at the step $m$, we choose an element \[u^m \in A \setminus (M(u^1) \cup \dots \cup M(u^{m-1}))\] provided that that set is not empty. Letting $t$ be the number of elements $u^i$ which are chosen by this process before it stops, we take $A’ = \{u^1, \dots, u^t\}$. By definition of $t$, the set $A$ is contained in the union \[M(u^1) \cup \dots \cup M(u^t),\] which has $M$-covering number at most $t|M|$, so $t \ge \Mc(A)/|M|$. By construction no two points of $A’$ can be contained in the same $M$-subspace: if that were the case for some points $u^i,u^j \in A’$ with $i<j$, then we would have $u^j \in M(u^i)$, contradicting the definition of $M(u^i)$.\end{proof}

\subsection{Deducing restrictions with linear bounds}

We now separately deduce Theorem \ref{linear bound for lattice subsets} and Theorem \ref{off-diagonal covering number} from Lemma \ref{subset where elements cannot be in the same subspace}.

\begin{proof}[Proof of Theorem \ref{linear bound for lattice subsets}]

Let $A$ be an order-$d$ lattice subset such that $\Mc(A) \ge |M|l$. By Lemma \ref{subset where elements cannot be in the same subspace} we can find a subset $A'$ of $A$ with size equal to $l$ and such that $\Mc(A') = |A'| = l$. For each $j \in [d]$ let $X_j = \{x_j: x \in A'\}$ be the projection of $A''$ on the $j$th coordinate axis. The sets $X_j$ all have size at most $|A'| = l$, and the set $A'$ is contained in $X_1 \times \dots \times X_d$, so $A \cap (X_1 \times \dots \times X_d)$ contains $A'$ and in particular \[\Mc(A \cap (X_1 \times \dots \times X_d)) \ge \Mc(A') = l.\qedhere \] \end{proof}

In order to obtain Theorem \ref{off-diagonal covering number}, we will use the following well-known statement, the proof of which is an instance of the probabilistic method which we recall for completeness. We note that Theorem \ref{off-diagonal covering number} specialises to Lemma \ref{Probabilistic method result} in the case $M = \{\emptyset\}$.

\begin{lemma} \label{Probabilistic method result} Let $d \ge 1$ be an integer, and let $A$ be a subset of $[n_1] \times \dots \times [n_d]$. We can find pairwise disjoint $X_1 \subset [n_1], \dots, X_d \subset [n_d]$ such that  \[|A(X_1 \times \dots \times X_d)| \ge |A \setminus E| / d^d.\] \end{lemma}

\begin{proof} We can assume without loss of generality that $n_1=\dots=n_d$ are all equal to some common value $n$. Indeed, if $A$ is a subset of $[n_1] \times \dots \times [n_d]$ then we can also view it as a subset of $[n]^d$, where $n =\max(n_1, \dots, n_d)$. We then define independent uniform random variables $U_1, \dots, U_n$ taking values in $[d]$, and then take \[X_j = \{i \in [n]: U_i=j\}\] for every $j \in [d]$. By construction the (random) sets $X_1, \dots, X_d$ are pairwise disjoint. For every $x \in A \setminus E$, the probability that $x$ belongs to $X_1 \times \dots \times X_d$ is equal to $1/d^d$, so the expected size of \[(A \setminus E)(X_1 \times \dots \times X_d) = A(X_1 \times \dots \times X_d)\] is $|A \setminus E|/d^d$. In particular, there is a desired choice of $(X_1, \dots, X_d)$ for which the required inequality holds. \end{proof}

Having shown that, we are now ready to deduce Theorem \ref{off-diagonal covering number}.

\begin{proof}[Proof of Theorem \ref{off-diagonal covering number}]

Let $A$ be an order-$d$ lattice subset such that \[\Mc(A \setminus E) \ge d^d|M|l.\] Without loss of generality we can assume that $A \cap E$ is empty. Again, by Lemma \ref{subset where elements cannot be in the same subspace} we can find a subset $A’$ of $A$ with $|A’| = d^d l$ and $\Mc A’ = d^d l$. Applying Lemma \ref{Probabilistic method result} we can find a subset $A''$ of $A’$ with size $|A’| / d^d$ and pairwise disjoint subsets $X_1 \subset [n_1], \dots,  X_d \subset [n_d]$ such that $A''$ is contained in $X_1 \times \dots \times X_d$. We then have $\Mc(A'') \ge |A''|$, as well as \[\Mc(A \cap (X_1 \times \dots \times X_d)) \ge l,\] as desired. \end{proof}

\section{Descending subspaces arguments} \label{Section: Descending subspaces arguments}

\subsection{A bounded size statement}

Our main auxiliary tool in this section is a result, Proposition \ref{no significant subspace implies bounded size}, which says in particular that if $A$ is an order-$d$ lattice subset such that every subspace of $[n_1] \times \dots \times [n_d]$ with order between $1$ and $d-1$ has an intersection with $A$ that is covered by a bounded number of lower-order subspaces, then $A$ has bounded size. Although its proof is technically simpler, the iterative step in that proof result is similar in spirit to \cite{K. subtensors}, Proposition 11.4. Likewise, Proposition \ref{no significant subspace implies bounded size} itself is analogous to a consequence of \cite{K. subtensors}, Proposition 11.4 that we obtain by applying it at most $2^d$ times: the fact that if an order-$d$ tensor $T$ has bounded $R$-rank, and all the order-$d’$ subtensors with $2 \le d’ \le d$ that we can obtain from $T$ by fixing some of the its $d$ coordinates have bounded partition rank, then $T$ has bounded tensor rank.

If $d \ge 1$ is an integer and $C$ is a non-empty subset of $[d]$ then we define a class $C^*$ of subsets of $[d]$ by \[C^* = \{C \setminus \{j\}: j \in C\}.\]

\begin{proposition} \label{no significant subspace implies bounded size} Let $d \ge 1$, $l \ge 1$ be integers, let $M \subset \cp([d])$ be non-empty, and let $A$ be an order-$d$ lattice subset. Assume that for every $B \in M$ and every non-empty $C \subset B$ the intersections of $A$ with every $C$-subspace have $C^*$-covering number at most $l$. Then $|A| \le l^{\tau} \Mc(A)$, where $\tau$ is the largest size of any $B \in M$. \end{proposition} 

\begin{proof}

If $M = \{\emptyset\}$ then $\Mc(A) = |A|$, so we are done. Let us now assume that $M \neq \{\emptyset\}$. Because $A$ has $M$-covering number at most $l$, we can write \begin{equation} A = \bigcup_{B \in M} \bigcup_{1 \le i \le t_B} A \cap S_{B,i} \label{first decomposition} \end{equation} for some $B$-subspaces $S_{B,i}$ and some nonnegative integers $t_B$ satisfying \[\sum_{B \in M} t_B = \Mc(A).\] Let $B_{\max} \in M$ with maximal size. By the assumption, for each $i \in [t_{B_{\max}}]$ we can write \[A \cap S_{B_{\max},i} = \bigcup_{B^1 \in B_{\max}^*} \bigcup_{1 \le i_1 \le u_{B^1}} A \cap S_{B^1,i_1}^1\] for some $B^1$-subspaces $S_{B^1,i_1}^1$ and some nonnegative integers $u_{B^1}$ satisfying \[\sum_{B^1 \in B_{\max}^*} u_{B^1} \le l.\] Letting $M^1 = (M \setminus \{B_{\max}\}) \cup B_{\max}^*$ we have obtained an analogue of \eqref{first decomposition} in terms of the sets in $M^1$, which we can write as \begin{equation} A = \bigcup_{B \in M^1} \bigcup_{1 \le i \le t_B^1} A \cap S_{B,i}^1 \label{second decomposition} \end{equation} where the subspaces $S_{B,i}^1$ are $M^1$-subspaces.

If $M^1 = \{\emptyset\}$ then every $A \cap S_{B,i}$ appearing in \eqref{second decomposition} has size at most $1$. Otherwise we iterate by selecting a new set $B_{\max}^2$ with maximal size in $M^1$ which plays the same role as $B_{\max} ^1 = B_{\max}$ previously did. More generally for any $m \ge 2$ the step $m$ involves a set $B_{\max}^m$ with maximal size in $M^{m-1}$ and leads to a decomposition \begin{equation} A = \bigcup_{B \in M^m} \bigcup_{1 \le i \le t_B^m} A \cap S_{B,i}^m \label{decomposition after step m} \end{equation} where $M^m = (M \setminus B_{\max}^m) \cup B_{\max}^m$ and the sets $S_{B,i}^m$ are $M^m$-subspaces. At the end of step $m$ we stop if $M^m = \{\emptyset\}$, and otherwise continue. 

For any step $m$ the set $B_{\max}^m$ in that iteration can no longer appear after the end of that step. Since the number of subsets of $[d]$ is bounded above (by $2^d$), this process eventually terminates, after at most $2^d$ iterations. Whenever at some step $m$ of the process a set $B_{\max}^m$ is selected, each subspace $S_{B_{\max}^{m-1},i}$ from the analogue at step $m-1$ of \eqref{decomposition after step m} is replaced by at most $l$ subspaces with strictly smaller order in \eqref{decomposition after step m}. Since the order of any subspace throughout the process can only take values between $0$ and $\tau$, so can be decreased only at most $\tau$ times, any subspace $S_{B,i}$ from \eqref{first decomposition} can have at most $l^{\tau}$ descendants with order $0$ throughout the process. The result follows. \end{proof}

\subsection{Deducing bounded numbers of decompositions}

The tool which we just obtained then allows us to prove Theorem \ref{number of decompositions} by induction on the $M$-covering number of the lattice subset $A$, by distinguishing into the cases where a subspace has a rich intersection with $A$, which then forces one of a few subspaces to belong to a minimal-length decomposition of $A$, and the case where there is no such subspace, in which case $A$ has bounded size and it is then not too difficult to conclude. In the case of tensors, there are likewise situations where we can guarantee that some specific functions will be involved in a minimal-length decomposition up to a change of basis (for instance \cite{K. decompositions}, Proposition 3.4), but we cannot do so in general. Furthermore, the analogue of the second case in the proof below is a tensor with bounded tensor rank, and showing that the set of minimal-length $R$-rank decompositions of a tensor cannot be too rich appears to be difficult even under the assumption that that tensor has bounded tensor rank.

\begin{proof}[Proof of Theorem \ref{number of decompositions}]

We write $Q(d,l,M)$ for our desired upper bound. Let $A$ be an order-$d$ lattice subset satisfying $\Mc(A)=l$. We proceed by induction on $l$. The result is immediate if $l=0$. Assume now that $l \ge 1$. We distinguish two cases.

\textbf{Case 1:} We can find a $C$-subspace $S$ with $C$ non-empty such that the $C^*$-covering number of $A \cap S$ is at least $l+1$. Then necessarily any $M$-covering decomposition of $A$ must necessarily involve a subspace containing $S$: if it did not, then writing a length-$l$ decomposition of $A$ and taking the intersection with $S$ would show that $A \cap S$ has $C^*$-covering number at most $l$. We may choose from at most $|M|$ possible $M$-subspaces $S’$ containing $S$. Because $\Mc(A)=l$, for each of these possibilities for $S'$ we have $\Mc(A \setminus S’) \ge l-1$. The inductive hypothesis shows that that for each of them, the lattice subset $A \setminus S’$ only has at most $Q(d,l-1,M)$ length-$(l-1)$ $M$-covering decompositions, so there are at most $l|M|Q(d,l-1,M)$ length-$l$ $M$-covering decompositions. (The factor $l$ comes from the fact that there are $l$ possibilities for how to place $S’$ in the $l$-tuple $(S_1, \dots, S_l)$ which constitutes the decomposition.)

\textbf{Case 2:} If we are not in Case 1, then the assumption of Proposition \ref{no significant subspace implies bounded size} is satisfied, and hence $|A| \le l^d$: if $M$ does not contain $[d]$ then this follows from $|A| \le l^{d-1}\Mc(A)$, since $\Mc(A) \le l$ and if $M$ contains $[d]$ then this follows from $|A| \le l^{d}\Mc(A)$, since $\Mc(A) \le 1$. There are hence at most $l^d|M|$ possible $M$-subspaces that do not have empty intersection with $A$. A length-$l$ decomposition is specified by an $l$-tuple of the subspaces, so $A$ has at most $(l^d|M|)^l$ length-$l$ $M$-covering decompositions. We have shown \[Q(d,l,M) \le \max(l|M|Q(d-1,l,M), l^{dl}|M|^l).\] We have $Q(1,l,M) \le l! \le l^l$, so we may take $Q(d,l,M) \le l^{dl} |M|^l$. \end{proof}

\subsection{Deducing restrictions with the same covering number}

Once Theorem \ref{number of decompositions} is proved, we deduce Theorem \ref{lattice subset with same covering number} by returning to the two cases which we have previously considered. In the first case, we restrict the subspace that had a rich intersection with $A$ so that the restriction still has a rich intersection with $A$ but the coordinate sets have bounded size; this then allows us to reduce to the case of an $M$-covering number one less than the original one. The second case is straightforward, since if a set has bounded size then all its coordinate projections have bounded size. As we do not have a proof of the analogue of Theorem \ref{number of decompositions} to the $R$-ranks of tensors we do not have such an analogue for Theorem \ref{lattice subset with same covering number} either.

\begin{proof}[Proof of Theorem \ref{lattice subset with same covering number}]

We write $J(d,l,M)$ for our desired upper bound. We proceed by induction on $d$. The result is immediate if $d=1$. We now assume that $d \ge 2$, and assume that the result has been proved for all $1 \le d' < d$. Let $A$ be an order-$d$ lattice subset satisfying $\Mc(A) \ge l$ for some nonnegative integer $l$. By removing points of $A$ one by one, we can without loss of generality assume that $\Mc A = l$.

We use the two cases from the proof of Theorem \ref{number of decompositions} to define a tree as follows. First take the root to be $A$. If we are in Case 2, then the construction stops. Otherwise, if we are in Case 1, then we fix a $C$-subspace $S$ satisfying Case 1, and for each of the at most $|M|$ possible $M$-subspaces $S’$ containing $S$ we take $A \setminus S’$ to be an immediate descendant of $A$. For each of these we then separately iterate: given a new set $A_v$, if it satisfies Case 2 then it becomes a leaf of the tree, otherwise we fix a $C_v$-subspace $S_v$ analogous to $S$, and then define its immediate descendants as we previously did. We continue the construction until every set of the tree is either declared to be a leaf or attains a depth equal to $l$.

For each non-leaf vertex $A_v$, the Case 1 assumption states that the $C_v^*$-covering number of $A_v \cap S_v$ is at least $l+1$.  Again, without loss of generality we can assume that the $C_v^*$-covering number of $A_v \cap S_v$ is exactly $l+1$. We know that $C_v$ is not the full of [d]: if it were, then this would contradict $\Mc A_v \le l$, since the $[d]^*$-covering number of any order-$d$ lattice subset is always at most its $M$-covering number. By the inductive hypothesis we may therefore apply Theorem \ref{lattice subset with same covering number} for order-$|C_v|$-lattice subsets to $A_v \cap S_v$, which shows that we can always find some sets $X_{1,v}, \dots, X_{d,v}$ with size at most $J(|C_v|, l+1, C_v^*)$ such that the $C_v^*$-covering number of \[(A_v \cap S_v)(X_{1,v} \times \dots \times X_{d,v})\] is at least $l+1$. 
 
For every leaf vertex $A_v$, Case 2 shows that $A_v$ has size at most $l^d$, so we can enclose it into a product of sets $X_{1,v} \times \dots \times X_{d,v}$ each with size at most $l^d$.

For every $j \in [d]$ we then take $X_j$ to be the union of the sets $X_{j,v}$ over all vertices $v$ of the tree. We next claim that \[\Mc A(X_1 \times \dots \times X_d) \ge l.\] Indeed, if we try to construct an $M$-covering of $A(X_1 \times \dots \times X_d)$ with smaller length, then Case 1 from the proof of Theorem \ref{number of decompositions} shows that necessarily we have to choose $M$-subspaces by going down the tree that we have just constructed, and once we arrive at a leaf, either the leaf already has depth $l$, or the remaining elements to incorporate are the same as those that we had for the original set $A$, and that cannot be done with fewer subspaces than for $A$ itself.

We may therefore take \[J(d,l,M) \le (1+|M|+\dots+|M|^l) \max_{C \in \cp([d]) \setminus \{\emptyset, [d]\}} J(d-1, l+1, C^*) + |M|^l l^d. \] We may take $J(1,l,M’) = l$ for all non-empty $M’$ (the only such $M’$ is $\{1\}$). As the $C^*$ contains at most $d$ sets, the bound follows. \end{proof}

\section{Further directions}\label{Section: Further directions}

Let us finish by mentioning a few questions motivated or inspired by the discussion in this paper.

In one direction, we can aim to improve and if possible exactly determine the optimal bounds in Theorem \ref{linear bound for lattice subsets} and in Theorem \ref{off-diagonal covering number}. The current upper bounds come from the lower bound on the $M$-independence number in terms of the $M$-covering number in Lemma \ref{subset where elements cannot be in the same subspace}.

\begin{question} \label{improving independence number bounds} Let $d \ge 1$ be an integer, and let $M \subset \cp([d])$ be non-empty. For which values of $c$ do we have $I_M(A) \ge c \Mc(A)$ for every order-$d$ lattice subset $A$ ? \end{question}

Even for simple cases of $M$, such as when $M$ consists in all subsets of $[d]$ with size $d-1$, this question does not appear to have an obvious answer. It is also possible that the optimal upper bounds in Theorem \ref{linear bound for lattice subsets} and in Theorem \ref{off-diagonal covering number} are better still than the optimal upper bounds coming from the current argument, even if Lemma \ref{subset where elements cannot be in the same subspace} on its own is optimised.

As we have seen many proofs are significantly or considerably easier for lattice covering numbers than for ranks of tensors, and the connection of Sawin and Tao (Proposition \ref{Sawin-Tao connection}) provides a translation between the two for sufficiently sparse tensors. The next question then naturally arises as a way of obtaining more statements and proofs similar to those of Corollary \ref{Corollary on the slice rank of sparse tensors}, which establish properties of sparse tensors purely by working with covering numbers of lattice subsets.

\begin{question} Let $d \ge 1$ be an integer, and let $R$ be a non-empty family of partitions of $[d]$. Are there any sparsity conditions on order-$d$ tensors which guarantee that for some non-empty $M \subset \cp([d])$ the $R$-rank of $T$ is equal to the $M$-covering number of the support of $T$ ?\end{question}

Proving such analogues of Proposition \ref{Sawin-Tao connection} does not appear obvious to achieve. The statement and proof of the dual formulation of the slice rank (\cite{Sawin and Tao}, Proposition 1) underpinning the Sawin-Tao connection fails for the partition rank. As for the tensor rank, the line covering number, corresponding to $M = \{\{j\}: j \in [d]\}$, may appear natural to consider but it will certainly not suffice for a tensor to be supported inside an antichain for its tensor rank to be equal to its line covering number: if it did, then that would answer a central open problem on tensor rank lower bounds by providing explicit examples of order-$d$ tensors $[n]^d \to \F$ with tensor rank $\Omega(n^{d-1})$.

\end{document}